\documentclass[english]{article}
\usepackage{amssymb,amscd,amsthm,amsmath,babel,amsfonts,young,wrapfig}
\usepackage{pict2e,cancel,pst-node,pst-text,pst-3d,pstricks}

\usepackage{graphicx}
\theoremstyle{plain}
\newtheorem{theorem}{Theorem}[section]

\newtheorem{corollary}[theorem]{Corollary}
\newtheorem{proposition}[theorem]{Proposition}

\theoremstyle{definition}
\newtheorem{notation}[theorem]{Notation}
\newtheorem{definition}[theorem]{Definition}
\newtheorem{remark}[theorem]{Remark}
\newtheorem{remarks}[theorem]{Remarks}

\newtheorem*{example*}{Example}
\newtheorem*{definition*}{Definition}

\numberwithin{equation}{section}

\newcommand\opn[2]{%
  \newcommand{#1}{\operatorname{#2}}}

\newcommand\NN{{\mathbb N}}
\newcommand\ZZ{{\mathbb Z}}

\opn\charac{char}
\opn\projdim{proj\,dim}
\opn\depth{depth}
\opn\rank{rank}
\opn\rlex{rlex}
\opn\LEX{Lex}
\opn\fine{end}
\opn\lcm{lcm}
\opn\modnuo{mod}
\opn\supp{supp}
\opn\Ass{Ass}
\opn\Proj{Proj}
\opn\Spec{Spec}
\opn\Soc{Soc}
\opn\reg{reg}
\opn\Md{Md}
\opn\Borel{Borel}
\opn\Shad{Shad}
\opn\Span{span}
\opn\ini{in}
\opn\Gin{Gin}
\opn\grade{grade}
\opn\length{length}
\opn\width{width}
\opn\greg{greg}
\opn\GL{GL}

\opn\Ker{Ker}
\opn\coker{coker}
 
\opn\im{Im}
\opn\Hom{Hom}
\opn\Tor{Tor}
\opn\Ext{Ext}
\opn\id{id}
\opn\Id{Id}
\opn\Homst{^*Hom}
\opn\Extst{^*Ext}
\opn\Gamst{^*\Gamma}
\opn\Hst{^*H}

\opn\KRS{KRS}
\opn\BKRS{BKRS}
\opn\ASL{ASL}
\opn\hgt{ht}

\opn\lk{lk}
\opn\st{st}

\let\:=\colon
\let\phi=\varphi

\newcommand\pnt{{\raise0.5mm\hbox{\large\bf.}}}

\newcommand\lra{\longrightarrow}
\newcommand\ra{\rightarrow}

\def\cocoa{\mbox{\rm
 C\kern-.13em o\kern-.07 em C\kern-.13em o\kern-.15em A}}
\opn\HF{H}
\opn\HS{HS}

\opn\po{pol}
\opn\pol{^{\bf p}} 

\newcommand{\Y}{\mathcal{Y}}
\newcommand{\X}{\mathcal{X}}
\newcommand{\Z}{\mathcal{Z}}

\newcommand\FF{{\mathbb F}}

\begin{document}
\noindent  
{\Large Invariants of ideals generated by pfaffians}

\vspace{1cm}
\noindent
{\bf Emanuela De Negri}. {\small Universit\`a di Genova, Dipartimento
  di Matematica, Via Dodecaneso 35, IT-16146 Genova, Italia. {\it
    email}: denegri@dima.unige.it}\\ 
\noindent
{\bf Elisa Gorla}. {\small  Universit\"at Basel, Departement Mathematik,
Rheinsprung 21, CH-4051 Basel, Switzerland.  {\it email}:
elisa.gorla@unibas.ch}

\vspace{.7cm}

\noindent
{\small {\bf Abstract}. Ideals generated by pfaffians are of interest
  in commutative algebra and algebraic geometry, as well as in
  combinatorics. In this article we compute multiplicity and
  Castelnuovo-Mumford regularity of pfaffian ideals of ladders. We
  give explicit formulas for some families of ideals, and indicate a
  procedure that allows to recursively compute the invariants of any
  pfaffian ideal of ladder. Our approach makes an essential use of
  liaison theory.
}

\vspace{1cm}
\noindent

\section*{Introduction}

Pfaffians are the natural analogue of minors when working with
skew-symmetric matrices.
Ideals generated by pfaffians are studied in the context of
commutative algebra and algebraic geometry, as well as in
combinatorics. Many are the reasons for such an interest, e.g., many
ideals generated by pfaffians are Gorenstein (see, e.g.,~\cite{KL} 
and~\cite{D}). Conversely, due to a famous result (\cite{BE}) of
Buchsbaum and Eisenbud, any Gorenstein ideal of height $3$ of a
polynomial ring over a field is generated by the maximal pfaffians of 
a suitable skew-symmetric matrix of homogeneous forms. Ideals
generated by pfaffians arise
naturally in algebraic geometry as, e.g., ideals of
pfaffians in a generic skew-symmetric matrix define Schubert cells in
orthogonal Grassmannians. Moreover, some Grassmannians are defined by
pfaffians, as well as some of their secant varieties.

In this article, we compute numerical invariants of pfaffian ideals of
ladders. Pfaffian ideals of ladders are, informally speaking, ideals
generated by pfaffians which only involve indeterminates in a ladder
of a skew-symmetric matrix of indeterminates. The size of the
pfaffians is allowed to vary in different regions of the ladder. This
family was introduced by the authors in~\cite{DGo}, and contains the
classically studied ideals of $2t$-pfaffians of a matrix or of a
ladder. It is a very large family, and a natural one to study from the
point of view of liaison theory, since all the ideals in this family
arise from ideals of $2t$-pfaffians in a ladder while performing
elementary G-biliaisons. In~\cite{DGo} we proved that these
ideals are prime, 
normal and Cohen-Macaulay. The main result of the paper was a proof
that any pfaffian ideal of ladder can be obtained from an ideal
generated by indeterminates via a finite sequence of ascending
G-biliaisons. In particular they are glicci, i.e., they belong to the
G-liaison class of a complete intersection. The G-biliaison steps were
described very explicitly. Therefore, as a biproduct, it is possible
to recursively compute numerical invariants of pfaffian ideals of
ladders such as the multiplicity, the Hilbert function, the
$h$-vector, as well as a graded free resolution. In some cases it is also
possible to compute the graded Betti numbers and in particular the
Castelnuovo Mumford regularity. Although it is possible to perform these
computations in any specific example, it is in general hard to produce
explicit formulas. In this paper, we derive explicit formulas for some
classes of pfaffian ideals of ladders.

The paper is organized as follows. In Section~1 we fix the notation
and define the classes that we study. We also recall the main result
of~\cite{DGo} on which our approach is based. In Section~2
we give explicit or recursive formulas for the multiplicity of the
ideals that we study. In
Theorem~\ref{prod} we give a simple numerical condition which forces
the multiplicity of a pfaffian ideal of ladder to decompose as the
product of the multiplicities of two pfaffian ideals relative to
subladders. In Section~3 we compute 
Castelnuovo-Mumford regularities. In Section~4 we show how to use our
approach to compute the graded Betti numbers of ideals of pfaffians of
maximal size of a generic skew-symmetric matrix. We also give a simple
proof that the $h$-vectors of these ideals are of decreasing
type. The ideals generated by pfaffians of maximal size of a generic
skew-symmetric matrix are Gorenstein ideals of height $3$, so the
results are well-known. However, we are able to give a very simple
proof, which can be easily specialized to any Gorenstein ideal of
height $3$.  

\vspace{.5cm}

{\small 
{\bf Acknowledgements}. The second author was supported by the
  Swiss National Science Foundation under grant no. 123393. 
  Part of this work was done while the authors were attending the
  conference ``PASI 2009 in Commutative Algebra and its Connections to
  Geometry, honoring Wolmer Vasconcelos'', which took place in Olinda
  (Brazil) in August 2009. The authors wish to thank the organizers,
  the speakers and the participants to the conference for the
  stimulating working environment that they created.
}

\section{Some classes of pfaffian ladder ideals}

Let $X=(x_{ij})$ be an $n\times n$ skew-symmetric matrix of indeterminates.
In other words, the entries
$x_{ij}$ with $i<j$ are indeterminates, $x_{ij}=-x_{ji}$ for
$i>j$, and $x_{ii}=0$ for all $i=1,...,n$.
Let $R=K[X]=K[x_{ij} \;|\; 1\leq i<j\leq n ]$ be the polynomial ring associated 
to $X$. 

\begin{definition}\label{ladd}
A {\em ladder} $\mathcal Y$ of $X$ is a subset of the set
 $\{(i,j)\in\NN^2 \;|\; 1\le i,j\le n\}$
with the following properties :
\begin{enumerate}
\item if $(i,j)\in {\mathcal Y}$ then $(j,i)\in {\mathcal Y}$,
\item if  $i<h,j>k$ and $(i,j),(h,k)$ belong to $\mathcal Y$, then 
$(i,k),(i,h),(h,j),(j,k)$ belong to $\mathcal Y$.
\end{enumerate}
\end{definition}

We do not assume that a ladder $\Y$ is connected, nor that $X$ is the 
smallest skew-symmetric matrix having $\Y$ as ladder. We can assume
without loss of generality that the ladder $\Y$ is symmetric.

It is easy to see that any ladder can be decomposed as a union of
square subladders 
\begin{equation}\label{decomp}
\Y=\X_1\cup\ldots\cup \X_s
\end{equation} 
where $$\X_k=\{(i,j)\;|\; a_k\le i,j \le b_k\},$$ for some  integers
$1\leq a_1\leq\ldots\leq a_s\leq n$ and $1\leq b_1\leq\ldots\leq
b_s\leq n$ such that $a_k<b_k$ for all $k$.
We say that $\Y$ is the ladder with {\em upper corners}
$(a_1,b_1),\ldots,(a_s,b_s)$, and that $\X_k$ is the square 
subladder of $\Y$ with upper outside corner $(a_k,b_k)$.
We allow two upper corners to have the same first or second
coordinate, but we assume that no two upper corners coincide.
Notice that with this convention a ladder does not have a unique
decomposition of the form (\ref{decomp}). In other words, a ladder
does not correspond uniquely to a set of upper corners
$(a_1,b_1),\ldots,(a_s,b_s)$. However, the upper corners determine the
subladders $\X_k$, hence the ladder $\Y$ according to (\ref{decomp}). 

Let $t$ be a positive integer. A $2t$-pfaffian is the pfaffian of a
$2t\times 2t$ submatrix of $X$.
Given a ladder $\mathcal Y$ we set $Y=\{x_{ij}\in X\;|\; (i,j)\in
{\mathcal Y},\; i<j\}$.
We let $I_{2t}(Y)$ denote the ideal generated by the set of the
$2t$-pfaffians of $X$ which involve only indeterminates of $Y$. In
particular $I_{2t}(X)$ is the ideal generated by the $2t$-pfaffians of
$X$. We regard all the ideals as ideals in $K[X]$.

Whenever we consider a ladder $\Y$, we assume that it comes with its
set of upper corners and the corresponding decomposition as a
union of square subladders as in (\ref{decomp}). 

The following family of ideals has been introduced and studied
in~\cite{DGo}:

\begin{definition}\label{ideal}
Let $\Y=\X_1\cup\ldots\cup \X_s$ be a ladder as in
Definition~\ref{ladd}.\newline
Let $X_k=\{x_{ij}\;|\; (i,j)\in\X_k,\; i<j\}$ for $k=1,\dots,s$.
Fix a vector ${\bf t}=(t_1,\ldots,t_s)$, ${\bf t}\in
\{1,\ldots,\lfloor\frac{n}{2}\rfloor\}^s$. 
The {\em pfaffian ideal} $I_{2{\bf t}}(Y)$ is by definition the sum of 
pfaffian ideals $I_{2t_1}(X_1)+\ldots+I_{2t_s}(X_s)\subseteq
K[X]$. We refer to these ideals as {\em pfaffian ideals of ladders}.  
\end{definition}

\begin{remarks}[Remarks~1.5,~\cite{DGo}]
We can assume without loss of generality that $$2t_k\leq
b_k-a_k+1,\;\;\;\mbox{for}\; 1\leq k\leq s.$$
Moreover, we can assume
that $$a_k-a_{k-1}>t_{k-1}-t_k \;\;\;\mbox{and}\;\;\;
b_k-b_{k-1}>t_k-t_{k-1}$$ for $2\leq k\leq s$.
\end{remarks}

In~\cite{DGo}, pfaffian ideals of ladders are proved to be prime,
normal, and Cohen-Macaulay. A formula for their height is given. 

\begin{notation}\label{ladderheight}
For a ladder $\Y$ with upper corners $(a_1,b_1),\ldots,(a_s,b_s)$ and
${\bf t}=(t_1,\ldots,t_s)$, we denote by $\tilde{\Y}$ the ladder with
upper corners $(a_1+t_1-1,b_1-t_1+1),\ldots,(a_s+t_s-1,b_s-t_s+1)$.
\end{notation}

The ladder $\tilde{\Y}$ computes the height of the ideal $I_{2{\bf t}}(Y)$
as follows:

\begin{proposition}[Proposition~1.10, \cite{DGo}]
Let $\Y$ be the ladder with upper corners $(a_1,b_1), \ldots,$
$(a_s,b_s)$ and ${\bf t}=(t_1,\ldots,t_s)$. Let $\tilde{\Y}$ be as in
Notation~\ref{ladderheight}. Then the height of $I_{2{\bf t}}(Y)$
equals the cardinality of $\{(i,j)\in\tilde{\Y} \;|\; i<j\}$. 
\end{proposition} 

We now recall the definition of biliaison.

\begin{definition}\label{gbil}
Let $I,I',J$ be homogeneous, saturated ideals in $K[X]$, with
$\hgt(I)=\hgt(I')=\hgt(J)+1.$ Assume that $R/J$ is Cohen-Macaulay and
generically Gorenstein, i.e., $(R/J)_P$ is Gorenstein for any minimal
associated prime $P$ of $J$. We say that $I$ is obtained from $I'$ by
a {\em G-biliaison of height $\ell$} on $J$ if $I/J$ and $I'/J(\ell)$
represent the same element in the ideal class group of $K[X]/J$. 
\end{definition}

In other words, $I$ is obtained from $I'$ by a G-biliaison of height
$\ell$ on $J$ if there exist homogeneous polynomials $f,g\in R$ with
$\deg(g)=\deg(f)+\ell$, such that $fI+J=gI'+J$ as ideals of $R$.

The main result of~\cite{DGo} is that ladder pfaffian ideals belong to the
G-biliaison class of a complete intersection. In particular, they are glicci. 
We briefly recall the single G-biliaison step which is described in the
proof of~\cite[Theorem 2.3]{DGo}. With the notation of
Definition~\ref{ideal}, let $\mathcal Y'$ be the subladder of
$\mathcal Y$ with upper corners
$$(a_1,b_1),\ldots, (a_{k-1},b_{k-1}), (a_k+1,b_k-1),
(a_{k+1},b_{k+1}),\ldots, (a_s,b_s),$$
and let ${\bf t'}=(t_1,\ldots,t_{k-1},t_k-1,t_{k+1},\ldots,t_s)$. Let
$\Z$ be the subladder of $\Y$ obtained by removing the entry
$(a_k,b_k)$ and its symmetric. Equivalently, $\Z$ is the ladder with
upper corners $$(a_1,b_1),\ldots, (a_{k-1},b_{k-1}), (a_k,b_k-1),
(a_k+1,b_k), (a_{k+1},b_{k+1}),\ldots, (a_s,b_s).$$ 
Let ${\bf u}=(t_1,\ldots,t_{k-1},t_k,t_k,t_{k+1},\ldots,t_s)$. One
has:

\begin{theorem}[Theorem~2.3, \cite{DGo}]\label{step}
Let $I=I_{2{\bf t}}(Y)$, $I'=I_{2{\bf t'}}(Y')$ and $J=I_{2{\bf
    u}}(Z)$ be ideals of $K[X]$. Then $I$ is obtained from $I'$ via an
elementary G-biliaison of height $1$ on $J$. 
\end{theorem}

More precisely, with the above notation we have
$$fI+J=gI'+J$$ where $f\in I'$ is a $2(t_k-1)$-pfaffian, $g\in I$ is a
$2t_k$-pfaffian, and $f,g\not\in J$.

When discussing biliaison, we will refer without distinction to the
ideals and to the varieties associated to them. 

In this paper we deal with special classes of pfaffian ideals of
ladders, and we compute some of their numerical invariants using the
biliaison step described in Theorem~\ref{step}. The same technique
gives a recursive procedure to determine such invariants for any
pfaffian ideal of ladder. However, it is in general hard to deduce
explicit formulas.

\medskip

We now introduce the classes we are going to study. First we consider
the ideal $L_t^n=I_{2{\bf t}}(Y)$ where $\Y$ is the ladder with upper
corners $(1,n-1)$ an $(2,n)$ and ${\bf t}=(t,t)$. Clearly $L_t^n$ is
generated by the $2t$-pfaffians of the ladder obtained from $X$ by
deleting the entries $(1,n)$ and $(n,1)$. 

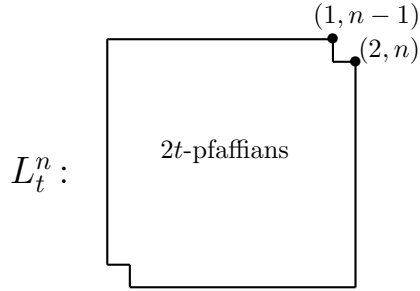
\begin{figure}[htbp]
\centering
\psset{xunit=.3cm, yunit=.3cm}
\begin{pspicture}(-2,1)(14,14)
\psline[linestyle=solid](1,12)(11,12)\rput(11,12){$\bullet$}\rput(12.5,13){$(1,n-1)$}
\psline[linestyle=solid](11,12)(11,11)
\psline[linestyle=solid](11,11)(12,11)\rput(12,11){$\bullet$}\rput(13.5,11.5){$(2,n)$}
\psline[linestyle=solid](12,11)(12,1)
\psline[linestyle=solid](1,12)(1,2)
\psline[linestyle=solid](1,2)(2,2)
\psline[linestyle=solid](2,2)(2,1)
\psline[linestyle=solid](2,1)(12,1)
\rput(6,7){ $2t$-pfaffians}
\rput(-2,6){\Large{$L_t^n$}\,:}
\end{pspicture}
\end{figure}

Then we restrict our attention to some ideals generated by pfaffians
whose size is maximal or submaximal, in a sense that we are going to
specify. 
In particular, we consider the ideals generated by maximal and by
submaximal pfaffians of a skew-symmetric matrix of
indeterminates. More precisely, we denote by $M_t$ the ideal generated 
by the $2t$-Pfaffians of a $(2t+1)\times(2t+1)$ matrix and by $SM_t$
the ideal generated by the $2t$-pfaffians of a $(2t+2)\times(2t+2)$
matrix. 

Moreover we consider ideals generated by pfaffians of two different
sizes in different regions of a matrix. Here we regard  {\it nested}
matrices as a ladder. In particular, we consider $N_t=I_{2{\bf t}}(Y)$
where $\Y$ is the ladder with upper corners $(1,2t-1)$ and $(1,2t+1)$,
and ${\bf t}=(t-1,t)$. So $N_t$ is the ideal generated by the
$2t$-pfaffians of a skew-symmetric matrix of size $2t+1$ and the
$(2t-2)$-pfaffians of its first $2t-1$ rows and columns. We denote by
$SN_t$ the ideal $I_{2{\bf t}}(Y)$ where $\Y$ is the ladder with upper
corners $(1,2t-1)$ and $(1,2t+2)$, and ${\bf t}=(t-1,t)$. This is the
ideal generated by the $2t$-pfaffians of a skew-symmetric matrix of
size $2t+2$ and the $(2t-2)$-pfaffians of its first $2t-1$ rows and
columns. 

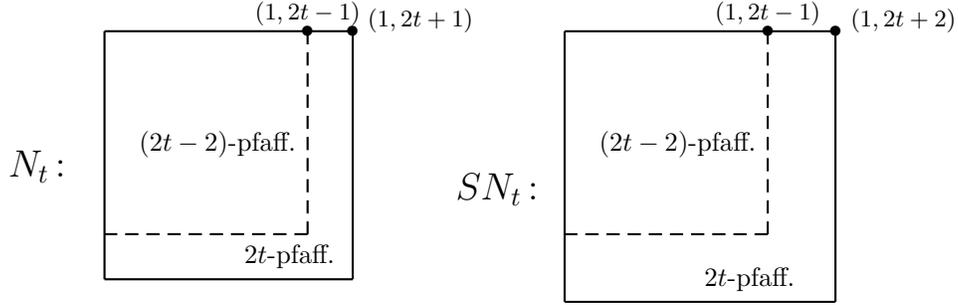
\begin{figure}[htbp]
\psset{xunit=.3cm, yunit=.3cm}
\begin{pspicture}(-2,-1)(14,14)
\psline[linestyle=solid](1,12)(12,12)
\rput(10,12){$\bullet$}\rput(10,12.8){\small{$(1,2t-1)$}}
\rput(12,12){$\bullet$}\rput(15,12.5){\small{$(1,2t+1)$}}
\psline[linestyle=solid](12,12)(12,1)
\psline[linestyle=solid](1,12)(1,1)
\psline[linestyle=solid](1,1)(12,1)
\psline[linestyle=dashed](10,12)(10,3)
\psline[linestyle=dashed](10,3)(1,3)
\rput(6,7){$(2t-2)$-pfaff.}
\rput(9,2){ $2t$-pfaff.}
\rput(-2,6){\Large{$N_t$}\,:}
\end{pspicture}
\hskip1.5truecm
\begin{pspicture}(-1,-0)(13,13)
\psline[linestyle=solid](1,13)(13,13)
\rput(10,13){$\bullet$}\rput(10,13.8){\small{$(1,2t-1)$}}
\rput(13,13){$\bullet$}\rput(16,13.5){\small{$(1,2t+2)$}}
\psline[linestyle=solid](13,13)(13,1)
\psline[linestyle=solid](1,13)(1,1)
\psline[linestyle=solid](1,1)(13,1)
\psline[linestyle=dashed](10,13)(10,4)
\psline[linestyle=dashed](10,4)(1,4)
\rput(6,8){$(2t-2)$-pfaff.}
\rput(9,2){ $2t$-pfaff.}
\rput(-2,6){\Large{$SN_t$}\,:}
\end{pspicture}
\end{figure}

We let $L_t(k)=I_{2{\bf t}}(Y)$, where $\Y$ is the ladder with upper
corners $(1,2t+1),(2,2t+2),(3,2t+3),\ldots,(k,2t+k)$, and ${\bf
  t}=(t,\dots, t)$. Notice that $L_t(1)=M_t$, and $L_t(2)=L_t^{2t+2}.$ 

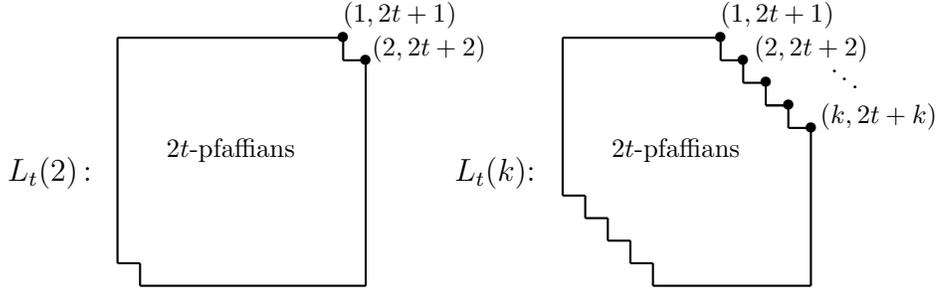
\begin{figure}[htbp]
\psset{xunit=.3cm, yunit=.3cm}
\begin{pspicture}(-2,-1)(14,14)
\psline[linestyle=solid](1,12)(11,12)\rput(11,12){$\bullet$}\rput(13.5,13){$(1,2t+1)$}
\psline[linestyle=solid](11,12)(11,11)
\psline[linestyle=solid](11,11)(12,11)\rput(12,11){$\bullet$}\rput(14.8,11.5){$(2,2t+2)$}
\psline[linestyle=solid](12,11)(12,1)
\psline[linestyle=solid](1,12)(1,2)
\psline[linestyle=solid](1,2)(2,2)
\psline[linestyle=solid](2,2)(2,1)
\psline[linestyle=solid](2,1)(12,1)
\rput(6,7){$2t$-pfaffians}
\rput(-2,6){\large{$L_t(2)$}\,:}
\end{pspicture}
\hskip1truecm
\begin{pspicture}(-2,-1)(14,14)
  \psline[linestyle=solid](1,12)(8,12)\rput(8,12){$\bullet$}\rput(10.5,13){$(1,2t+1)$}
  \psline[linestyle=solid](8,12)(8,11)
  \psline[linestyle=solid](8,11)(9,11)\rput(9,11){$\bullet$}\rput(12,11.5){$(2,2t+2)$}
  \psline[linestyle=solid](9,11)(9,10)
  \psline[linestyle=solid](9,10)(10,10)\rput(10,10){$\bullet$}
  \psline[linestyle=solid](10,10)(10,9)
  \psline[linestyle=solid](10,9)(11,9)\rput(11,9){$\bullet$}\rput(13.5,10.5){$\ddots$}
  \psline[linestyle=solid](10,9)(11,9)
  \psline[linestyle=solid](11,9)(11,8)
  \psline[linestyle=solid](11,8)(12,8)\rput(12,8){$\bullet$}\rput(15,8.5){$(k,2t+k)$}
  \psline[linestyle=solid](12,8)(12,1)
  \psline[linestyle=solid](1,12)(1,5)
  \psline[linestyle=solid](1,5)(2,5)
  \psline[linestyle=solid](2,5)(2,4)
  \psline[linestyle=solid](2,4)(3,4)
  \psline[linestyle=solid](3,4)(3,3)
  \psline[linestyle=solid](3,3)(4,3)
  \psline[linestyle=solid](4,3)(4,2)
  \psline[linestyle=solid](4,2)(5,2)
  \psline[linestyle=solid](5,2)(5,1)
  \psline[linestyle=solid](5,1)(12,1) \rput(6,7){$2t$-pfaffians}
  \rput(-2,6){\large{$L_t(k)$}:}
\end{pspicture}
\end{figure}

Moreover, given two integers $j$ and $k$ we let $\Y_{jk}$ be the ladder with
the $j+k$ upper outside corners
$(1,2t-1),(2,2t),(3,2t+1),\dots,(j,2t+j-2),(j,2t+j),(j+1,2t+j+1),\dots,(j+k-1,2t+j+k-1).$
We consider the ideal
$$L_t(j,k):=I_{2{\bf t}}(Y_{jk}), \   \mbox{  where \  \  }{\bf t}
=(\underbrace{t-1,\dots,t-1}_{j},\underbrace{t, \dots,t}_k).$$  
Notice that $L_t(0,k)=L_{t+1}(k,0)$. Moreover, this class contains
most of the classes that we have already introduced. More precisely:
$L_t(k)=L_t(0,k)$, $M_t=L_t(0,1)$, $SM_t=L_t(1,0)$, and
$N_t=L_t(1,1)$.  

\begin{figure}[htbp]
\centering
\psset{xunit=.5cm, yunit=.5cm}
\begin{pspicture}(-2,-1)(14,14)
\psline[linestyle=solid](1,12)(5,12)\rput(5,12){$\bullet$}\rput(7,12.5){$(1,2t-1)$}
\psline[linestyle=solid](5,12)(5,11)
\psline[linestyle=solid](5,11)(6,11)\rput(6,11){$\bullet$}\rput(7,11.5){$(2,2t)$}
\psline[linestyle=solid](6,11)(6,10)
\psline[linestyle=solid](6,10)(9,10)\rput(7,10){$\bullet$}\rput(7.5,10.5){\small{$(3,2t+1)$}}
\rput(9,10){$\bullet$}\rput(10.8,10.5){$(3,2t+3)$}
\psline[linestyle=solid](9,10)(9,9)
\psline[linestyle=solid](9,9)(10,9)\rput(10,9){$\bullet$}\rput(11.5,9.5){$(4,2t+4)$}
\psline[linestyle=solid](10,9)(10,8)
\psline[linestyle=solid](10,8)(11,8)\rput(11,8){$\bullet$}\rput(12.8,8.3){$(5,2t+5)$}
\psline[linestyle=solid](11,8)(11,7)
\psline[linestyle=solid](11,7)(12,7)\rput(12,7){$\bullet$}\rput(13.7,7.5){$(6,2t+6)$}
\psline[linestyle=solid](12,7)(12,1)
\psline[linestyle=solid](1,12)(1,8)
\psline[linestyle=solid](1,8)(2,8)
\psline[linestyle=solid](2,8)(2,7)
\psline[linestyle=solid](2,7)(3,7)
\psline[linestyle=solid](3,7)(3,4)
\psline[linestyle=solid](3,4)(4,4)
\psline[linestyle=solid](4,4)(4,3)
\psline[linestyle=solid](4,3)(5,3)
\psline[linestyle=solid](5,3)(5,2)
\psline[linestyle=solid](5,2)(6,2)
\psline[linestyle=solid](6,2)(6,1)
\psline[linestyle=solid](6,1)(12,1)
\psline[linestyle=dashed](7,10)(7,6)
\psline[linestyle=dashed](7,6)(3,6)
\rput(9,4){\large $2t$-pfaffians}
\rput(3.8,9){\large$(2t-2)$-pfaff.}
\rput(-2,6){\Large{$L_t(j,k)$ }\,:}
\rput(-2,4){\large{$j=3, k=4$}}
\end{pspicture}
\end{figure}

Given two integers $j$ and $k$, we let $\Z_{jk}$ be the ladder with
the $j+k$ upper outside corners
$(1,2t-1),(2,2t),(3,2t+1),\dots,(j,2t+j-2),(j+1,2t+j+1),\dots,(j+k,2t+j+k).$
We consider the ideal
$$H_t(j,k):=I_{2{\bf t}}(Z_{jk}), \   \mbox{  where \  \  }{\bf t}
=(\underbrace{t-1,\dots,t-1}_{j},\underbrace{t, \dots,t}_k).$$ 
It is $L_t(k)=H_t(0,k)=H_{t+1}(k,0)$.

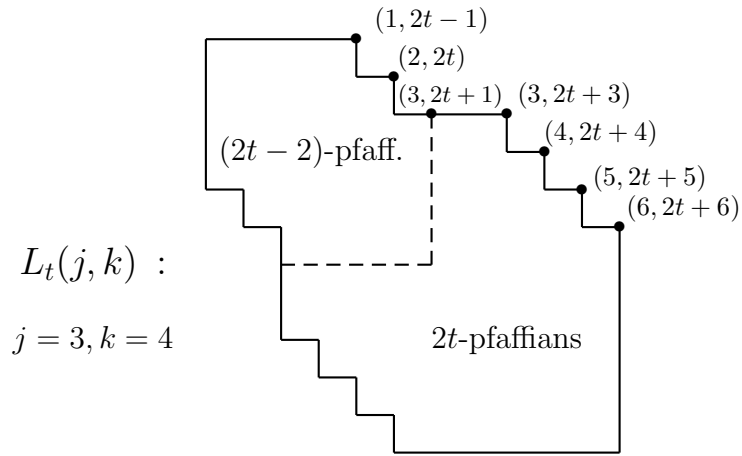
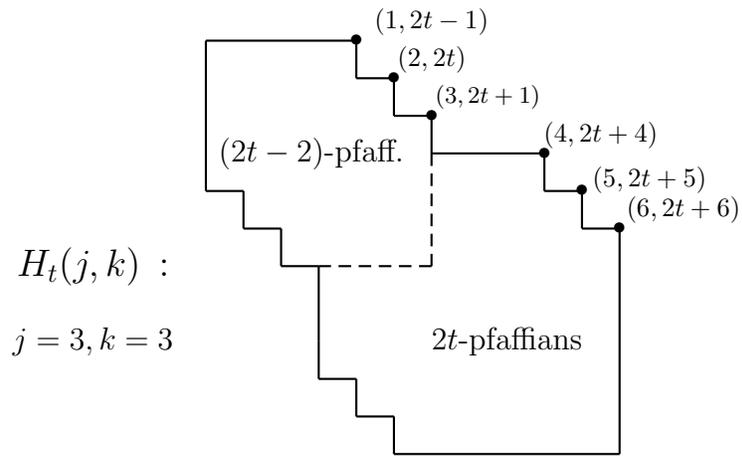
\begin{figure}[htbp]
\centering
\psset{xunit=.5cm, yunit=.5cm}
\begin{pspicture}(-2,-1)(14,14)
\psline[linestyle=solid](1,12)(5,12)\rput(5,12){$\bullet$}\rput(7,12.5){$(1,2t-1)$}
\psline[linestyle=solid](5,12)(5,11)
\psline[linestyle=solid](5,11)(6,11)\rput(6,11){$\bullet$}\rput(7,11.5){$(2,2t)$}
\psline[linestyle=solid](6,11)(6,10)
\psline[linestyle=solid](6,10)(7,10)\rput(7,10){$\bullet$}\rput(8.5,10.5){\small{$(3,2t+1)$}}
\psline[linestyle=solid](7,10)(7,9)
\psline[linestyle=solid](7,9)(10,9)\rput(10,9){$\bullet$}\rput(11.5,9.5){$(4,2t+4)$}
\psline[linestyle=solid](10,9)(10,8)
\psline[linestyle=solid](10,8)(11,8)\rput(11,8){$\bullet$}\rput(12.8,8.3){$(5,2t+5)$}
\psline[linestyle=solid](11,8)(11,7)
\psline[linestyle=solid](11,7)(12,7)\rput(12,7){$\bullet$}\rput(13.7,7.5){$(6,2t+6)$}
\psline[linestyle=solid](12,7)(12,1)
\psline[linestyle=solid](1,12)(1,8)
\psline[linestyle=solid](1,8)(2,8)
\psline[linestyle=solid](2,8)(2,7)
\psline[linestyle=solid](2,7)(3,7)
\psline[linestyle=solid](3,7)(3,6)
\psline[linestyle=solid](3,6)(4,6)
\psline[linestyle=solid](4,6)(4,4)
\psline[linestyle=solid](4,4)(4,3)
\psline[linestyle=solid](4,3)(5,3)
\psline[linestyle=solid](5,3)(5,2)
\psline[linestyle=solid](5,2)(6,2)
\psline[linestyle=solid](6,2)(6,1)
\psline[linestyle=solid](6,1)(12,1)
\psline[linestyle=dashed](7,9)(7,6)
\psline[linestyle=dashed](7,6)(4,6)
\rput(9,4){\large $2t$-pfaffians}
\rput(3.8,9){\large$(2t-2)$-pfaff.}
\rput(-2,6){\Large{$H_t(j,k)$ }\,:}
\rput(-2,4){\large{$j=3, k=3$}}
\end{pspicture}
\end{figure}


\section{Multiplicity of pfaffian ladder ideals}
 
In this section we give some formulas for  the multiplicity of the
ideals introduced in the previous section.  
Throughout the section, we denote by $e(I)$ the multiplicity of
$R/I$ for any ideal $I\subset R=K[X]$. All the formulas that we 
produce are obtained as a finite sum of positive
contributions. Therefore they are well suited to give lower bounds for
the multiplicity. 
In the sequel we will need the following well know fact, which we
prove for completeness. 
 
\begin{proposition}\label{liaisonMultiplicity}
Let $H,I,J\subset K[X]$ be homogeneous, saturated, unmixed
ideals. Assume that $H$ is Cohen-Macaulay and that $I$ is obtained
from $J$ via an elementary G-biliaison of height $\ell\in\ZZ$ on
$H$. Then $$e(I)=e(J)+\ell e(H).$$ 
\end{proposition}

\begin{proof}
Let $U,S,T$ be the schemes associated to $H,I,J$, respectively. Under
our assumptions, $U$ is arithmetically Cohen-Macaulay and $S,T$ are
generalized divisors on $U$. Moreover, $S$ is linearly equivalent to
$T+\ell h$ as generalized divisors on $U$, where $h$ denotes the 
hyperplane section class on $U$. In particular 
$$e(I)=\deg(S)=\deg(T)+\ell\deg(U)=e(J)+\ell e(H).$$
\end{proof}

We denote by $I_{t}^n$ the ideal generated by the $2t$-pfaffians of an
$n\times n$ skew-symmetric matrix of indeterminates. 
In \cite[Theorem 7]{K} Krattenthaler proved that
\begin{equation}\label{krattenthaler}
e(I_{t}^n )=\prod_{1\le i \le j \le n-2t+1}
\frac{2(t-1)+i+j}{i+j}.
\end{equation} 
In particular for the ideals $M_t$ and $SM_t$ one has:
$$e(M_t) =\prod_{1\le i \le j \le 2}
\frac{2(t-1)+i+j}{i+j}, \  \   \   e(SM_t) =\prod_{1\le i\le j \le 3}
\frac{2(t-1)+i+j}{i+j}.$$

From the results in~\cite{DGo} one can easily deduce a formula for the
multiplicity of the ideal $L_t^n$.

\begin{proposition}
$$\mbox{$e(L_t^n)=$\Large $ \frac{(n-2t+2)!}{(2n-4t+4)!}
\left[\frac{(2n-2t+2)!}{n!}-\frac{(n-1)!}{(2t-3)!}\right]$ }
\prod_{1\le i\le j\le n-2t+2} \frac{2(t-1)+i+j}{i+j}$$ 
\end{proposition}
 
\begin{proof}
By Theorem~\ref{step} the ideal $I_t^{n+1}$ is obtained from
$I_{t-1}^{n-1}$ via an elementary G-biliaison of height $1$ on
$L_t^n$. Hence by Proposition~\ref{liaisonMultiplicity}
$$e(L_t^n)=e(I_t^{n+1})-e(I_{t-1}^{n-1}).$$ Substituting
(\ref{krattenthaler}) we obtain 
$e(L_t^n)=$
$$
\displaystyle\prod_{1\le i \le j \le n-2t+2} \frac{1}{i+j}
\Big[\displaystyle\prod_{1\le i \le j \le n-2t+2}(2t-2+i+j)-
\displaystyle\prod_{1\le i \le j \le
  n-2t+2}(2t-4+i+j)\Big].
$$
Since  $$\displaystyle\prod_{1\le i \le j \le
  n-2t+2}(2t-4+i+j)=
\prod_{0\le i \le j \le n-2t+1}(2t-2+i+j)$$
by means of direct computation one gets 
$$\begin{array}{l} 
\displaystyle\prod_{1\le i \le j \le
  n-2t+2}(2t-2+i+j)-\displaystyle\prod_{1\le i \le j \le
  n-2t+2}(2t-4+i+j)= \\ 
\displaystyle\prod_{1\le i \le j \le n-2t+1}
(2(t-1)+i+j)\Big[\displaystyle\prod_{1\le i\le
  n-2t+2}(n+i)-\displaystyle\prod_{0\le j\le n-2t+1}(2t-2+j)\Big]
\end{array}.$$
The result now follows from the equality
$$\begin{array}{l} 
\displaystyle\prod_{1\le i \le j \le n-2t+2} \frac{1}{i+j}\displaystyle\prod_{1\le i \le j \le n-2t+1}
(2(t-1)+i+j)= \\
\displaystyle\prod_{1\le i \le j \le n-2t+1}
\frac{(2(t-1)+i+j)}{i+j}\displaystyle\prod_{1\le i \le n-2t+2}
\frac{1}{n-2t+2+i}. 
\end{array}$$
\end{proof}
 
The case of ideals generated by maximal pfaffians of a matrix has been
extensively studied. In particular it is well known that 
\begin{equation}\label{gorcod3}
e(L_t(1))=e(M_t)=1+2^2+3^2+\cdots+t^2
\end{equation}
(see~\cite[Section~6]{HTV}, and~\cite[Theorem~5.6 and the following
example]{HT}).

We deduce the following formulas from Theorem~\ref{step}. 

\begin{proposition}\label{f_2(t)} 
$$e(L_t(2))=1+\sum_{s=2}^t [ 2s(1+2^2+\cdots +s^2)-s^3]$$ and 
$$e(N_t)= 1+\sum_{s=2}^{t-1}[2s(1+2^2+\cdots +s^2)-s^3]+t(1+2^2+\cdots +(t-1)^2).$$
\end{proposition}
 
\begin{proof}
By Theorem~\ref{step} the ideal $L_t(2)$ is obtained from $N_t$
via an elementary G-biliaison of height $1$ on $M_t+(f)$, where $f$ is
a $2t$-pfaffian which is regular modulo $M_t$. Thus by
Proposition~\ref{liaisonMultiplicity} one has 
\begin{equation}\label{lt2_first}
e(L_t(2))=e(N_t)+e(M_t+(f))=e(N_t)+te(M_t).\end{equation} 
Moreover the ideal $N_t$ is obtained from  $L_{t-1}(2)$ via an
elementary G-biliaison of height $1$ on $M_{t-1}+(g)$, where $g$ is a
$2t$-pfaffian which is regular modulo $M_{t-1}$. Therefore 
\begin{equation}\label{combin}
e(N_t)=e(L_{t-1}(2))+te(M_{t-1})
\end{equation} 
and combining (\ref{lt2_first}) and (\ref{combin}) one gets 
\begin{equation}\label{lt2}
e(L_t(2))=e(L_{t-1}(2))+te(M_{t-1})+te(M_t).
\end{equation}
Finally by (\ref{lt2}) and (\ref{gorcod3}), after solving the recursion
one obtains 
$$e(L_t(2))=1+\sum_{s=2}^t [s(e(M_{s-1})+e(M_s))]=1+\sum_{s=2}^t
[2s(1+2^2+\cdots +s^2)-s^3].$$ The formula for $e(N_t)$ follows from
substituting the formula for $e(L_t(2))$ and (\ref{gorcod3}) in
(\ref{combin}).  
\end{proof}

We now deduce a formula for the multiplicity of ideals generated by
submaximal pfaffians.

\begin{corollary}\label{e(t)}
$$e(SM_t)=t+\sum_{r=2}^t\sum_{s=2}^r [ 2s(1+2^2+\cdots +s^2)-s^3].$$
\end{corollary}
 
\begin{proof}
Since $SM_t$ is obtained from $SM_{t-1}$ via an elementary G-biliaison
of height $1$ on $L_t(2)$, one has $e(SM_t)=e(SM_{t-1})+e(L_t(2))$. By
solving the recursion and using Proposition \ref{f_2(t)}, one obtains
the result.
\end{proof}

Let $\Y=\Y_1\cup\Y_2$ be a ladder which is union of two smaller
ladders. Let $I_1=I_{2{\bf t_1}}(Y_1)$ and $I_2=I_{2{\bf t_2}}(Y_2)$ be pfaffian
ideals associated to the ladders $\Y_1$ and $\Y_2$, and let the upper
corners of $\Y$ be the union of the upper corners of $\Y_1$ and
$\Y_2$. Let ${\bf t}={\bf t_1}\oplus {\bf t_2}$ be the vector obtained
by appending the 
vector ${\bf t_2}$ to the vector ${\bf t_1}$ and let $I=I_{2t}(Y)=I_1+I_2$ be the
pfaffian ideal associated to the ladder $\Y$. If
$\Y_1\cap\Y_2=\emptyset$, one can easily show that 
\begin{equation}\label{product}
e(I)=e(I_1)e(I_2).
\end{equation} 
The following theorem gives a sufficient condition on the ladder so
that (\ref{product}) holds.

\begin{theorem}\label{prod}
Let $\Y,\Y_1,\Y_2$ be ladders, $\Y=\Y_1\cup\Y_2$. Let
$I_1=I_{2{\bf t_1}}(Y_1)$ and $I_2=I_{2{\bf t_2}}(Y_2)$ be pfaffian ideals of
ladders associated to $\Y_1$ and $\Y_2$. Let ${\bf t}={\bf t_1}\oplus
{\bf t_2}$ and let $I=I_{2{\bf t}}(Y)=I_1+I_2$ be the corresponding
pfaffian ideal of 
ladder. Let $\tilde{\Y},\tilde{\Y_1},\tilde{\Y_2}$ be defined as in
Notation~\ref{ladderheight}, and let
$\tilde{Y},\tilde{Y_1},\tilde{Y_2}$ be the corresponding sets of
indeterminates. If
$\tilde{Y_1}\cap\tilde{Y_2}=\emptyset$, then $$e(I)=e(I_1)e(I_2).$$
\end{theorem}

\begin{proof}
Let $\Z=\Y_1\cap\Y_2$, $R_1=K[Y_1]/I_1$, and
$R_2=K[Y_2]/I_2$. We have $$K[Y]/I\cong
R_1\otimes_K R_2/J$$ where $J$ is generated by $|Z|$ linear forms
(which identify the corresponding indeterminates in $Y_1$ and
$Y_2$). If $\tilde{Y_1}\cap\tilde{Y_2}=\emptyset$, then $$\hgt I=\hgt I_1+\hgt
I_2$$ hence $$\hgt J=\dim R_1\otimes R_2-\dim K[Y]/I=|Y_1|-\hgt
I_1+|Y_2|-\hgt I_2-|Y|+\hgt I=|Z|.$$ Since $R_1\otimes R_2$ is a
Cohen-Macaulay ring, $J$ is generated by a regular sequence and
$$e(I)=e(R_1\otimes_K R_2/J)=e(I_1)e(I_2).$$
\end{proof}

We now give an example of a family of pfaffian ideals of ladders whose
multiplicity can be computed directly from~Theorem~\ref{prod}.

\begin{proposition}\label{htjk}
$$e(H_t(j,k))=e(L_{t-1}(j))e(L_t(k)).$$
\end{proposition}

\begin{proof}
Let $\Y=\Z_{jk}$ be the ladder with the $j+k$ upper  corners
$(1,2t-1),(2,2t),$ $\ldots,(j,2t+j-2),(j+1,2t+j+1),\ldots,(j+k,2t+j+k).$
Let $\Y_1$ be the ladder with the $j$ upper 
corners $(1,2t-1),\ldots,(j,2t+j-2)$ and let
$\Y_2$ be the ladder with the $k$ upper  corners
$(j+1,2t+j+1),\ldots,(j+k,2t+j+k).$ Clearly $\Y=\Y_1\cup\Y_2$. Let 
$${\bf t_1}=(\underbrace{t-1,\dots,t-1}_{j}),\;
{\bf t_2}=(\underbrace{t,\dots,t}_k),\; {\bf t}={\bf t_1\oplus
  t_2}=(\underbrace{t-1,\dots,t-1}_{j},\underbrace{t, \dots,t}_k).$$ 
Then $\tilde{\Y_1}$ is the ladder with upper outside corners
$(t-1,t+1),(t,t+2),\ldots,(t+j-2,t+j)$ and $\tilde{\Y_2}$ is the ladder
with upper outside corners $(t+j,t+j+2),\ldots,(t+j+k-1,t+j+k+1).$ Hence 
$\tilde{\Y_1}\cap\tilde{\Y_2}=
\{(t+j,t+j)\}$ and $Y_1\cap Y_2=\emptyset$. By Theorem~\ref{prod} it
follows that $$e(H_t(j,k))=e(I_1)e(I_2)$$ where $I_1=I_{2{\bf
    t_1}}(Y_1)$ and $I_2=I_{2{\bf t_2}}(Y_2)$. The thesis follows from
the observation that $I_1=L_{t-1}(j)$ and $I_2=L_t(k)$. 
\end{proof}

Combining Proposition~\ref{htjk} and Proposition~\ref{f_n(t)}, we
obtain a formula for the multiplicity of the ideals $L_t(j,k)$.

\begin{proposition}\label{ltjk}
For $j,k\geq 1$ we have
$$e(L_t(j,k))=e(L_{t-1}(j+k))+te(L_{t-1}(j+k-1))+\sum_{l=1}^{k-1}e(L_{t-1}(j+k-1-l))e(L_t(l)).$$
\end{proposition}
 
\begin{proof}
We proceed by induction on $k\geq 1$. By Theorem~\ref{step},
$L_t(j,1)$ is obtained from 
$L_{t-1}(j+1)$ via an elementary G-biliaison on $L_{t-1}(j)+(f)$, where
$f$ is a $2t$-pfaffian which does not belong to $L_{t-1}(j)$. Hence by
Proposition~\ref{liaisonMultiplicity} 
$$e(L_t(j,1))=e(L_{t-1}(j+1))+te(L_{t-1}(j)).$$
This proves the thesis for $k=1$.
To establish the formula for $k\geq 2$, observe that $L_t(j,k)$ is
obtained from $L_t(j+1,k-1)$ via an elementary 
G-biliaison of height $1$ on $H_t(j,k-1)$. Hence by
Proposition~\ref{liaisonMultiplicity} and Proposition~\ref{htjk}
\begin{equation}\label{lh}
e(L_t(j,k))=e(L_t(j+1,k-1))+e(L_{t-1}(j))e(L_t(k-1)).
\end{equation} 
By induction hypothesis $e(L_t(j+1,k-1))=$ 
$$e(L_{t-1}(j+k))+te(L_{t-1}(j+k-1))+
\sum_{l=1}^{k-2}e(L_{t-1}(j+k-1-l))e(L_t(l))$$
and the thesis follows.
\end{proof}

Explicit formulas for $e(L_t(1))$ and $e(L_t(2))$ were given in
(\ref{gorcod3}) and in Proposition~\ref{f_2(t)}. Since $L_1(k)$ is
generated by indeterminates, $e(L_1(k))=1$ for any $k$.
The following formula allows us to calculate $e(L_t(k))$ recursively,
for $t\geq 2$ and $k\ge 3$. 

\begin{proposition}\label{f_n(t)}
For $t,k\geq 2$ we have $e(L_t(k))=$
$$e(L_{t-1}(k))+t[e(L_t(k-1))+e(L_{t-1}(k-1))]+
\sum_{l=1}^{k-2}e(L_{t-1}(k-1-l))e(L_t(l)).$$  
\end{proposition}

\begin{proof}
By Theorem \ref{step}, $L_t(k)$ is obtained from $L_t(1,k-1)$
via an elementary G-biliaison of height $1$ on $L_t(k-1)+(f)$, where $f$
is a $2t$-pfaffian which does not belong to $L_t(k-1)$. Hence by
Proposition~\ref{liaisonMultiplicity} and Proposition~\ref{ltjk}
$$e(L_t(k))=L_t(1,k-1)+te(L_t(k-1))=$$ 
$$e(L_{t-1}(k))+t[e(L_t(k-1))+e(L_{t-1}(k-1))]+\sum_{l=1}^{k-2}e(L_{t-1}(k-1-l))e(L_t(l)).$$
\end{proof}

\begin{remarks}
\begin{enumerate}
\item Proposition~\ref{f_n(t)} allows us to compute the
  multiplicity of the ideals $L_t(k)$ for any values of $t$ and
  $k$. This can in fact be done recursively, using as a starting point
  that $e(L_1(k))=1$ for any $k$, and the explicit formulas for the
  multiplicities of $L_t(1)=M_t$ and $L_t(2)$ which appear in
  (\ref{gorcod3}) and in Proposition~\ref{f_2(t)}, respectively.
\item Proposition~\ref{ltjk} allows us to compute the multiplicity of
  the ideals $L_t(j,k)$ for any values of $t,j,k$. One can in fact use
  Proposition~\ref{f_n(t)} to compute the multiplicities of
  $L_t(1),\ldots,L_t(k-1)$ and $L_{t-1}(j),\ldots,L_{t-1}(j+k)$.
\item Since $L_t(k)=L_t(0,k)$, the multiplicity of $L_t(j,k)$ for
$j=0$ is computed in Proposition~\ref{f_n(t)}. In fact, the formula
obtained in Proposition~\ref{f_n(t)} corresponds to the formula
computed in Proposition~\ref{ltjk} for $j=0$, taken ``cum grano
salis''.
\item The formula given in Proposition~\ref{ltjk} is false for $k=0$.
\end{enumerate}
\end{remarks}

Finally, we express the multiplicity of $SN_t$ in terms of the
multiplicities of $SM_t$ and $L_t(1,2)$. The latter two can be computed
by Proposition~\ref{e(t)} and Proposition~\ref{ltjk}.

\begin{proposition}\label{snt}
For $t\geq 1$ we have
$$e(SN_t)=\sum_{s=2}^t e(L_s(1,2))+\sum_{s=2}^{t-1} s\,e(SM_s)+1.$$
\end{proposition}

\begin{proof}
We proceed by induction on $t$. If $t=1$, then $SN_1$ is generated by
indeterminates and $e(SN_1)=1$.

Let $\Y$ denote the ladder with upper corners $(1,2t-1)$ and
$(2,2t+1)$. Then $I_{2(t-1)}(Y)$ is the ideal generated by the
$2(t-1)$-pfaffians of $\Y$. By Theorem~\ref{step}, $SN_t$ is
obtained from $I_{2(t-1)}(Y)$ via an elementary G-biliaison of
height $1$ on $L_t(1,2)$. In turn, $I_{2(t-1)}(Y)$ is obtained from
$SN_{t-1}$ via an elementary G-biliaison of height $1$ on
$SM_{t-1}+(f)$, where $f$ is a $2(t-1)$-pfaffian which does not belong
to $SM_{t-1}$. Therefore, by Proposition~\ref{liaisonMultiplicity} 
$$e(SN_t)=e(L_t(1,2))+(t-1)e(SM_{t-1})+e(SN_{t-1})$$ 
and the thesis follows by induction hypothesis.
\end{proof} 

\begin{remark}
From the proof of Proposition~\ref{snt} it also follows that 
$$e(I_{2(t-1)}(Y))=e(SN_t)-e(L_t(1,2))=\sum_{s=2}^{t-1}[e(L_s(1,2)+se(M_s)]+1.$$
\end{remark}

\section{Castelnuovo-Mumford regularity}

In this section we use biliaison to compute the Castelnuovo-Mumford
regularity of some of the ideals considered in the
previous section. For an ideal $I$ of $R=K[X]$, we denote by
$\beta_{i,j}(I) $ the $(i,j)-$th graded Betti number of $I$, regarded
as an $R$-module.  
The Castelnuovo-Mumford regularity of a Cohen-Macaulay ideal $I$ of
height $h=\hgt(I)$ is
$$\reg(I)=\max\{ j\mid\beta_{h-1,j}(I)\neq 0\}-h+1.$$
It is well known that $\reg(M_t)=2t-1$.

The following result allows us to recursively compute the
Castelnuovo Mumford regularities of ideals obtained one from the other
by biliaison.

\begin{theorem}\label{liaisonRegularity}
Let $H,I,J\subset R$ be homogeneous, Cohen-Macaulay ideals. Assume that
$I$ is obtained from $J$ via an elementary G-biliaison of height
$\ell\in\ZZ$ on $H$. If $\reg(J)<\reg(H)$,
then $$\reg(I)=\reg(H)+\ell-1.$$ 
\end{theorem}

\begin{proof}
Since $I$ is obtained from $J$ via an elementary G-biliaison of height
$\ell\in\ZZ$ on $H$, there are homogeneous polynomials $f,g$ with
$\deg(f)+\ell=\deg(g)=:t$ such that $$fI+H=gJ+H\subset R.$$ Let
$h=\hgt I=\hgt J=\hgt H+1$. Applying
the Mapping Cone construction to the short exact sequence
$$0\lra H[-t]\lra H\oplus J[-t] \lra gJ+H\lra 0$$ we have that 
$$\reg(gJ+H)=\max\{j\mid \beta_{h-1,j}(gJ+H)\neq
0\}-h+1=$$ $$\max\{\reg(H)+h-2,\reg(J)+h-1\}+t-h+1=\reg(H)+t-1.$$ 
The last equality follows from the assumption that
$\reg(J)<\reg(H)$. The previous equality follows from the observation
that, since $J$ and $H$ are Cohen-Macaulay ideals,
$$\max\{j\mid \beta_{h-2,j}(H)\neq 0\}=\reg(H)+h-2\geq$$ 
$$\reg(J)+h-1>\max\{j\mid \beta_{h-2,j}(J)\neq 0\}$$
therefore no cancellation involving a direct summand $R[-reg(H)+h-2]$
can take place in the free resolution of $gJ+H$. 

In an analogous fashion, we can produce a free resolution for
$gJ+H=fI+H$ by applying the Mapping Cone construction to the short
exact sequence 
$$0\lra H[-t+\ell]\lra H\oplus I[-t+\ell] \lra fI+H\lra 0.$$ 
Since
$$\max\{j\mid \beta_{h-1,j}(fI+H)\neq 0\}=\reg(H)+t+h-2>$$
$$\reg(H)+h-2+t-\ell=\max\{j\mid \beta_{h-2,j}(H[-t+\ell])\neq 0\},$$
it must be 
$$\reg(H)+t+h-2=\max\{j\mid \beta_{h-1,j}(I[-t+\ell])\neq 0\}=
\reg(I)+h-1+t-\ell,$$
hence $$\reg(I)=\reg(H)+\ell-1.$$
\end{proof}

We now derive formulas for the Castelnuovo-Mumford regularity of some
pfaffian ideals of ladders. They are all easy consequences of
Theorem~\ref{liaisonRegularity}. 

\begin{proposition}\label{lt2_reg} 
For $t\geq 1$ we have
$$\reg(L_t(2))=3t-2$$ and for $t\geq 2$
$$\reg(N_t)=3t-4.$$
\end{proposition}

\begin{proof}
We compute the regularity of $L_{t-1}(2)$ and $N_t$ for $t\geq 2$. We
proceed by induction on $t\geq 2$.
If $t=2$, $L_1(2)$ is generated by indeterminates, hence
$\reg(L_1(2))=1.$ By Theorem \ref{step}, $N_2$ is obtained from
$L_1(2)$ via an ascending G-biliaison of height $1$ on $M_1+(p)$,
where $p$ is a $4$-pfaffian which is regular modulo
$M_1$. Since $\reg(L_1(2))=1<2=\reg(H)$, by
Theorem~\ref{liaisonRegularity} we have $$\reg(N_2)=2.$$

We now assume by induction hypothesis that $\reg(L_{t-2}(2))=3t-8$ and
$\reg(N_{t-1})=3t-7$, and compute the regularity of $L_{t-1}(2)$ and
$N_t$. By Theorem \ref{step}, the ideal $L_{t-1}(2)$ is obtained
from $N_{t-1}$ via an elementary G-biliaison of height $1$ on
$M_{t-1}+(f)$, where $f$ is a $2(t-1)$-pfaffian which is regular modulo
$M_{t-1}$. Since $\reg(N_{t-1})=3t-7<3t-5=\reg(M_{t-1}+(f))$, by
Theorem~\ref{liaisonRegularity} $$\reg(L_{t-1}(2))=3t-5.$$
By Theorem \ref{step}, the ideal $N_t$ is obtained from
$L_{t-1}(2)$ via an elementary G-biliaison of height $1$ on
$M_{t-1}+(g)$, where $g$ is a $2t$-pfaffian which is 
regular modulo $M_{t-1}$. Since
$reg(L_{t-1}(2))=3t-5<3t-4=\reg(M_{t-1}+g)$,
by Theorem~\ref{liaisonRegularity} we have $$\reg(N_t)=3t-4.$$ 
\end{proof}

\begin{proposition}
For $t\geq 1$ we have
$$\reg(SM_t)=3t-2.$$
\end{proposition}

\begin{proof}
We proceed by induction on $t\geq 1$. If $t=1$, $SM_1$ is generated by
indeterminates, hence $\reg(SM_1)=1.$ 
By Theorem \ref{step} the ideal $SM_t$ is obtained from
$SM_{t-1}$ via an elementary G-biliaison of height $1$ on $L_t(2)$. By
induction hypothesis and Proposition~\ref{lt2_reg} 
$$\reg(SM_{t-1})=3t-5<3t-2=\reg{L_t(2)}.$$ Therefore, by 
Theorem~\ref{liaisonRegularity} $$\reg(SM_t)=\reg(L_t(2))=3t-2.$$
\end{proof}

\section{The Gorenstein height $3$ case}

The ideal $M_t$ generated by the $2t$-pfaffians of a generic
skew-symmetric matrix of size $2t+1$ is a Gorenstein ideal of 
height $3$. A classical result due to Buchsbaum and Eisenbud
\cite{BE} states that any Gorenstein ideal of height $3$ is obtained
by specialization from $M_t$, for some $t$. 
An alternative proof for many classically known results on Gorenstein
ideals of height $3$ can therefore be given by combining
specialization with a liaison approach analogous to what we have done
in the previous sections. 
 
In this section we wish to give a taste of what can be obtained
following such an approach. In particular, we use G-biliaison to
compute the graded Betti numbers of the ideal $M_t$ and to prove that
its $h$-vector is of decreasing type. We start by recalling some
definitions and fixing the notation.

Let $I$ be a homogeneous ideal of $R=K[X]$. The {\em Hilbert function}
of $R/I$ is defined as 
$$\HF_I(m)=\dim_K(R/I)_m$$ for every integer $m$. Clearly $\HF_I(m)=0$
for $m<0$. 
The formal power series $$\HS_I(z)=\sum_{m\in\ZZ}\HF_I(m)z^m$$  is
called the {\em Hilbert series} of $R/I$. 
It is well known that the Hilbert series of $R/I$ is of the form 
$$\HS_I(z)=\frac {h_I(0)+h_I(1)z+\ldots+h_I(s) z^s}{(1-z)^d},$$ where
$d=\dim (R/I)$ and $h_i\in\ZZ$ for every $i$. The
vector $$h_I=(h_I(0),\dots,h_I(s))\in \ZZ^s$$ is called {\em
  $h$-vector} of $I$. 
Moreover we denote by $\Delta H_I$ the {\em first difference} of $H_I$, that is
$$\Delta H_I(m)= H_I(m)-H_I(m-1).$$ 
\begin{definition}
Let $h=(h_0,h_1,\dots,h_s)\in {\bf Z}^s$
\item{a)}  $h$ is {\em unimodal} if there exists $t\in\{1,\dots s\}$ such that
 $h_1\le h_2 \le \dots \le h_t\ge h_{t+1}\ge \dots \ge h_s$.
 \item{b)}  $h$ is of {\em decreasing type} if whenever $h_t>
   h_{t+1},$ then $h_j>h_{j+1}$ for every $j>t$. 
\end{definition}
Notice that every $h$-vector of decreasing type is unimodal.

\begin{proposition}
The $h$-vector of $M_t$ is of decreasing type.
\end{proposition}

\begin{proof} 
Let $X$ be a $(2t+1)\times (2t+1)$ skew-symmetric matrix of
indeterminates and let $R=K[X]$ be the corresponding polynomial
ring. Denote by $h_{(t,t)}(m)$ the $m$-th entry of the $h$-vector
of a complete intersection generated by two forms of degree $t$. 

We follow the notation of Section~1 and consider the ideals $M_t$,
$M_{t-1}$ and $L_t^{2t+1}=:I$. It is
clear that $I$ is generated by two $2t$-pfaffians which form a
complete intersection. 
By Theorem~\ref{step}, $M_{t-1}$ is obtained from $M_t$ via an
elementary G-biliaison of height $1$ on $I$. In other words, there are
homogeneous polynomials $f,g$ of degree $t-1$ and $t$ respectively,
such that $$fM_t+I=gM_{t-1}+I\subset R.$$ By the additivity of the
Hilbert function on the two short exact sequences 
$$0\lra I[-t+1]\lra I\oplus M_t[-t+1] \lra fM_t+I\lra 0$$ 
$$0\lra I[-t]\lra I\oplus M_{t-1}[-t] \lra gM_{t-1}+I \lra 0$$ 
one obtains that $H_{M_t}(d-t+1)-H_I(d-t+1)=H_{M_{t-1}}(d-t)-H_I(d-t)$
for any $d\in\ZZ$. By setting $m=d-t+1$, we get
$$H_{M_t}(m)=H_{M_{t-1}}(m-1)+\Delta H_I(m).$$ Since $\dim R/I-1=\dim
R/M_{t-1}=\dim R/M_t$, one has
$$h_{M_t}(m)=h_{M_{t-1}}(m-1)+h_{(t,t)}(m).$$ Solving the recursion
one obtains  
$$h_{M_t}(m)=h_{M_1}(m-t+1)+\sum_{j=2}^{t}h_{(j,j)}(m-t+j).$$
This proves that the $h$-vector of $M_t$ is obtained  by summing the
$h$-vectors of suitable complete intersections. Notice that the
$h$-vectors involved in the summation are shifted in such a way,
that the maximum is always attained at the same point.
Therefore, their sum $h_{M_t}$ is of decreasing type.
\end{proof}

We can easily compute the graded Betti numbers of $M_t$ as follows.

\begin{proposition}
A minimal free resolution of $M_t$ has the form
$$0\lra R[-2t-1]\lra R[-t-1]^{2t+1}\lra R[-t]^{2t+1}\lra M_t\lra 0.$$
\end{proposition}
 
\begin{proof}
We prove the statement by induction on $t\geq 1$. If $t=1$, the ideal
$M_1$ is generated by three distinct indeterminates, hence a minimal
free resolution has the form 
$$0\lra R[-3]\lra R[-2]^3\lra R[-1]^3\lra M_1\lra 0.$$
Assume now that $t\geq 2$ and consider the ideals $M_t$, $M_{t-1}$ and
$L_t^{2t+1}$. We denote $L_t^{2t+1}$ by $I$ for brevity. It is clear
that $I$ is generated by two $2t$-pfaffians which form a complete
intersection.  
By Theorem~\ref{step}, $M_{t-1}$ is obtained from $M_t$ via an
elementary G-biliaison of height $1$ on $I$. Moreover, there are
homogeneous polynomials $f,g$ of degree $t-1$ and $t$ respectively,
such that $$fM_t+I=gM_{t-1}+I\subset R.$$ By induction hypothesis
$M_{t-1}$ has a minimal free resolution of the form
$$0\lra R[-2t+1]\lra R[-t]^{2t-1}\lra R[-t+1]^{2t-1}\lra M_{t-1}\lra
0.$$
Let $$0\lra\FF_3\lra\FF_2\lra\FF_1\lra M_t\lra 0$$ be a minimal free
resolution of $M_t$.
Applying the Mapping Cone to the two short exact sequences 
$$0\lra I[-t+1]\lra I\oplus M_t[-t+1] \lra fM_t+I\lra 0$$ 
$$0\lra I[-t]\lra I\oplus M_{t-1}[-t] \lra gM_{t-1}+I \lra 0$$ 
one obtains free resolutions for the ideal $J=fM_t+I=gM_{t-1}+I$ of
the form 
$$\begin{array}{ccccc}
 & R[-3t+1] & & R[-t]^2 & \\
0\lra & \oplus & \lra R[-2t]^{2t+2}\lra & \oplus & \lra J\lra 0 \\
 & R[-3t] & & R[-2t+1]^{2t-1} &
\end{array}$$ and
$$\begin{array}{ccccccc}
 &  R[-3t+1] & &  R[-2t]\oplus R[-2t+1]^2 & & R[-t]^2 & \\
0\ra & \oplus & \ra & \oplus & \ra & \oplus & \ra J\ra 0.\\
 & \FF_3[-t+1] & & \FF_2[-t+1] & &  \FF_1[-t+1] &
\end{array}$$
The first free resolution must be minimal, hence 
\begin{equation}\label{mfr}\FF_3\supseteq
R[-2t-1], \FF_2\supseteq R[-2t]^{2t+1}, \mbox{ and } \FF_1\supseteq
R[-t]^{2t+1}.\end{equation} 
Since no cancellation is possible among $\FF_1[-t+1],\FF_2[-t+1]$
and $\FF_3[-t+1]$ in the second free resolution of $J$, we deduce that
all the containments in (\ref{mfr}) must be equalities.
\end{proof}

\end{document}